\documentclass[twoside,11pt]{amsart}

\numberwithin{equation}{section}
\usepackage{amsmath}
\usepackage{amssymb,amsfonts,amsthm,latexsym,psfrag}
\usepackage{graphicx,color,epic,eepic,bbold,enumerate}
\usepackage[all]{xy}
\usepackage{hyperref}

\theoremstyle{plain}
        \newtheorem{theorem}{Theorem}[section]
        \newtheorem{lemma}[theorem]{Lemma}
        \newtheorem{proposition}[theorem]{Proposition}
        \newtheorem{corollary}[theorem]{Corollary}
               
\theoremstyle{definition}
        \newtheorem{definition}[theorem]{Definition}
        \newtheorem{remark}[theorem]{Remark}
        \newtheorem{example}[theorem]{Example}

\newcommand{\C}{\mathbb{C}}
\newcommand{\R}{\mathbb{R}}
\newcommand{\Q}{\mathbb{Q}}
\newcommand{\K}{\mathbb{K}}

\newcommand{\N}{\mathbb{N}}
\newcommand{\Z}{\mathbb{Z}}
\newcommand{\calA}{\mathcal{A}}

\newcommand{\calC}{\mathcal{C}}
\newcommand{\calD}{\mathcal{D}}

\newcommand{\calL}{\mathcal{L}}

\newcommand{\calN}{\mathcal{N}}
\newcommand{\calO}{\mathcal{O}}

\newcommand{\calP}{\mathcal{P}}

\newcommand{\calJ}{\mathcal{J}}
\newcommand{\calI}{\mathcal{I}}

\newcommand {\id} {\operatorname{id}}

\newcommand{\ov}{\overline}

\begin{document}
\title[]{Whitney functions determine the real homotopy type of a semi-analytic set}

\author{Bryce Chriestenson and Markus Pflaum}
\address{Mathematisches Institut, Ruprecht-Karls-Universit\"at, Im Neuenheimer Feld 288, \newline
         \indent 69120 Heidelberg, Germany}
\email{bchriestenson@mathi.uni-heidelberg.de} 
\address{Department of Mathematics, University of Colorado, Boulder, CO 80309-0395, U.S.A.}   
\email{Markus.Pflaum@colorado.edu}

\newcommand{\cat}[ 1 ] {\mathsf{#1}}
\newcommand{ \Ckf }[ 2 ]{ \ensuremath{ \calC^{ #2 }\left( #1 \right) } }
\newcommand{ \Dep }[ 1 ]{ \ensuremath{ \mathrm{ dp }_{ \mathcal{ Z } } \left( #1 \right) } }
\newcommand{ \Der }[ 2 ]{ \ensuremath{ \mathrm{ Der }_{ \RR } \left( #1 , #2 \right) } } 
\newcommand{ \Dim }[1]{ \ensuremath{ \mathrm{dim} \hspace{ 3pt } #1 } }
\newcommand{ \Dcx }[ 1 ]{ \ensuremath{ \Omega^{ \bullet }_{ \mathrm{ dR } } \left( #1 \right) } }
\newcommand{ \Df }[ 2 ]{ \ensuremath{ \Omega^{ #2 }_{ \mathrm{ dR } } \left( #1 \right) } }
\newcommand{ \Ds }[ 2 ]{ \ensuremath{ \Omega^{ #2 }_{  #1  } } }
\newcommand{ \Hd }[ 2 ]{ \ensuremath{ \mathrm{ H }_{ dR }^{ #2 } \left( #1 \right) } }
\newcommand{ \Hr }[ 3 ]{ \ensuremath{ \mathrm{ H }_{W}^{ #3 } \left( #1 , #2 \right) } }
\newcommand{ \HS }[ 3 ]{ \ensuremath{ \mathrm{ H }^{ #3 } \left( #1 , #2 \right) } }
\newcommand{ \Hw }[ 2 ]{ \ensuremath{ \mathrm{ H }_{ W }^{ #2 } \left( #1 \right) } }
\newcommand{\HH}[ 2 ]{ \ensuremath{ \mathrm{HH}_{ #2 } \left( #1 \right) } }
\newcommand{ \Hom }[ 2 ]{ \ensuremath{ \mathrm{ Hom }_{ \RR } \left( #1 , #2 \right) } }
\newcommand{ \Shom }[ 3 ]{ \ensuremath{ \mathrm{ Hom }_{ \Sf{ #3 } } \left( #1 , #2 \right) } }
\newcommand{ \Jet }[ 2 ]{ \ensuremath{ \mathrm{ J }^{ #2 } \left( #1 \right) } }
\newcommand{ \Jf }[ 2 ]{ \ensuremath{ \mathcal{ J }^{ \infty }\left( #1 , #2 \right) } }
\newcommand{ \Jfk }[ 3 ]{ \ensuremath{ \mathcal{ J }^{ #3 }\left( #1 , #2 \right) } }
\newcommand{ \JJ }{ \ensuremath{ \mathcal{ J }^{ \infty } } } 
\newcommand{ \Js }[ 2 ]{ \ensuremath{ \mathcal{ J }^{ \infty }_{ #1 , #2 } } } 
\newcommand{ \Jks }[ 3 ]{ \ensuremath{ \mathcal{ J }^{ #3 }_{ #1 , #2 } } }
\newcommand{ \KK }{\ensuremath{ \mathbb{K} }}
\newcommand{ \Map }[ 3 ]{ \ensuremath{ #1 : #2 \rightarrow #3 } }
\newcommand{ \NI }[ 2 ]{ \ensuremath{ \mathcal{ N }_{#2} \left( #1 \right) } }
\newcommand{ \Rdf }[ 3 ]{ \ensuremath{ \Omega^{ #3 } \left( #1 , #2 \right) } }
\newcommand{ \Rs }[ 2 ]{ \ensuremath{ \left( | #1 | , \mathcal{ #2 }_{#1} \right) } }
\newcommand{ \Sf }[ 1 ]{ \ensuremath{ \mathcal{ C }^{ \infty } \left( #1 \right) } }
\newcommand{ \Sco }[ 3 ]{ \ensuremath{ C^{ #3 } \left( #1 , #2 \right) } }
\newcommand{ \Ss }[ 1 ]{ \ensuremath{ \mathcal{ C }^{ \infty }_{ #1} } }
\newcommand{ \Shs }[ 1 ]{ \ensuremath{ \mathbb{S}_{#1 } }}
\newcommand{ \st }{ \ensuremath{ \hspace{ 3pt } \vline \hspace{ 3pt } } }
\newcommand{ \sWs }[ 2 ]{ \ensuremath{ \Gamma \left( #1 , \Ws{ #2 } \right) } }
\newcommand{ \sWds }[ 2 ]{ \ensuremath{ \Gamma \left( #1 , \Wds{ #1 }{ #2 } \right) }}
\newcommand{ \Vf }[ 1 ]{ \ensuremath{ \mathcal{ X } \left( #1 \right) } }
\newcommand{ \Wdf }[ 2 ]{ \ensuremath{ \Omega^{ #2 }_{ \mathrm{ W } } \left( #1 \right) } }
\newcommand{ \Wds }[ 2 ]{ \ensuremath{ \Omega^{ #2 }_{W, #1 } } }
\newcommand{ \Wcx }[ 1 ]{ \ensuremath{ \Omega^{ \bullet }_{ \mathrm{ W } } \left( #1 \right) } }
\newcommand{ \Wk }[ 2 ]{ \ensuremath{  \mathcal{ E }^{#2 }_{ #1 } } } 
\newcommand{ \Ws }[ 1 ]{ \ensuremath{ \mathcal{ E }^{ \infty }_{ #1 } } } 
\newcommand{ \Wf  }[ 1 ]{ \ensuremath{ \mathcal{ E }^{ \infty } \left( #1 \right) } }
\newcommand{ \WF  }[ 2 ]{ \ensuremath{ \mathcal{ E }^{ #2 } \left( #1 \right) } }

\newcommand{\qism}{
  \overset{\hspace{-0.8 em}{{}_\sim}}{\longrightarrow}
}
\newcommand{\lqism}{
  \overset{\hspace{0.8 em}{{}_\sim}}{\longleftarrow}
}
%%% Local Variables:
%%% mode: latex
%%% TeX-master: "ChrPfl"
%%% End:

\keywords{}
\subjclass[]{}
\begin{abstract}
In this paper, we investigate the Whitney--de Rham complex $\Wdf{X}{\bullet}$ associated to a semi-analytic subset $X$ of an 
analytic manifold $M$.  This complex is a commutative differential graded algebra, that is defined to be the quotient of the 
de Rham complex of smooth differential forms on $M$ by the differential graded ideal generated by all smooth functions which 
are flat on $X$.  We use Hironaka's desingularization theorem to prove a Poincar\'e Lemma for $\Wdf{X}{\bullet}$ holds true, 
which entails that its cohomology is isomorphic to the real cohomology of $X$.  Furthermore, we show that this isomorphism is 
induced by a quasi-isomorphism of differential graded algebras.  Thus it preserves the product structure, and is therefore an 
isomorphism of commutative differential graded algebras.  As a consequence we show, when $X$ is simply connected, that the 
Whitney--de Rham complex determines the real homotopy type of $X$.  This allows one further to conclude that the Hochschild
homology of the differential graded algebra $\Wdf{X}{\bullet}$ is isomorphic to the cohomology of the free loop space $\calL X$.
\end{abstract}
\maketitle
\tableofcontents
\bibliographystyle{amsplain}
\section*{Introduction}
Rational homotopy theory goes back to the fundamental work by Quillen 
\cite{quillen} and Sullivan \cite{sullivan1977infinitesimal} from the late 
60's and 70's.  The idea is to consider rational homotopy equivalence rather
than homotopy equivalence.  Two simple-connected CW-complexes $X$, and $Y$ are rationally equivalent
if there is a continuous map $f:X\rightarrow Y$ such that the induced map 
	$$\pi_k\left(f\right):\pi_k\left(X\right)\otimes \Q \rightarrow \pi_k\left(Y\right)\otimes \Q,$$
is an isomorphism for all $k$.  This condition is much weaker than than the tradional 
notion of homotopy equivalence.  For example each even dimensional real projective space, $\R P ^{2n}$, $n\geq 0$, 
is rationally equivalent to a point via the cannonical map to the point.  Thus one can see that a large amount of 
homotopy information is lost by considering only rational homotoy types.  The benefit is that things now become 
computable.  The classical Whitehead-Serre theorm already shows how much computational power is gained by ignoring 
the torsion
\begin{theorem}[Whitehead-Serre]
Let $f:X\rightarrow Y$ be a continuous map between simply connected CW-complexes then the following are equivalent
\begin{itemize}
\item $\pi_k\left(f\right):\pi_k\left(X\right)\otimes \Q \rightarrow\pi_k\left(Y\right)\otimes \Q$ is an isomorphism 
       for all $k$, 
\item $H_k\left(f\right):H_k\left(X,\Q\right)\rightarrow H_k\left(Y,\Q\right)$ is an isomorphism for all $k$.
\end{itemize}
\end{theorem}

This shows that the rational homotopy type is computable via only homological methods.  

A fundamentally new and in addition the most computational friendly approach to rational homotopy theory was 
developed by Sullivan \cite{sullivan1977infinitesimal}, using minimal models of commutative differential 
graded algebras (CDGA). 
The key idea is to associate to a simply connected CW-complex, $X$, the CDGA, $A_{PL}\left(X\right)$, of simplicial 
polynomial differential forms on $X$, and using this associate a minimal Sullivan algebra $M_X$ to 
$A_{PL}\left(X\right)$.  This is done in such a way that 
$$
 H^\bullet \left( M_X\right)\cong H^\bullet \left(A_{PL}\left(X\right)\right) \cong  H^\bullet\left(X,\Q\right).
$$
One observes that $M_X$ is uniquely determined by the rational homotopy type of $X$.  From here Sullivan notes that 
$M_X$ is actually a free CDGA $SV_X$ on a graded vector space $V_X$, where
$\left(V_X\right)_k \cong \mathrm{Hom}_{\Z}\left(\pi_k\left(X\right),\Q\right)$, if $X$ is of finite type.  
The benefit of $M_X = SV_X $ is that it is frequently computable, even by hand.  

In case the underlying space is a smooth manifold $M$, the  CDGA $A_{PL}\left(M\right)$ 
is quasi-isomorphic to the de Rham complex $\Dcx{M}$. As a consequence of this, the real homotopy type 
of $M$ is determined by the de Rham complex, which then led to the fundamental observation that the 
cohomology of formal manifolds such as for example K\"ahler manifolds determines the real homotopy type,
see \cite{deligne1975real}.

It is well-known by the work of Herrera and Bloom \cite{herrera,bloom-herrera} (see also Example 
\ref{Ex:BloomHerrera}) that over a semi-analytic set $X\subset \R^n$ the CDGA of smooth differential 
forms does in general not recover the real cohomology of $X$. In this paper, we propose a new idea
to remedy this deficit. Our approach  is to associate to each semi-analytic set $X$ the CDGA of 
Whitney--de Rham 
differential forms.  We proceed to show that this CDGA is quasi-isomorphic to $A_{PL}\left(X\right)$, and 
hence determines the real-homotopy type of $X$.  As an application, we prove, by further elaborating 
on work by Goodwillie \cite{goodwillie}, Burghelea-Fiedorowicz \cite{burghelea-fiedorowicz}, and 
Jones \cite{jones},
that the homology of the free loop space over a semi-analytic set $X$ is naturally isomorphic to the 
Hochschild homology of the CDGA of Whitney--de Rham  forms over $X$.

As an outlook let us mention that if the underlying set $X$ is subanalytic, the argument provided in 
Section \ref{Section:MainResults} to prove Main Theorem 2 together with the result from 
Brasselet--Pflaum \cite{brasselet-pflaum} that the cohomology of the Whitney--de Rham
complex over a subanalytic set coincides with its (real) singular cohomology, shows that 
even for subanalytic $X$ the real homotopy type is determined by the Whitney--de Rham complex. 
The proof from \cite{brasselet-pflaum} that the Whitney--de Rham cohomology of a subanalytic set $X$
recovers its  singular cohomology is only significantly more complicated than the
proof presented here for the case of a semianalytic set $X$. Work on simplifying that proof 
from \cite{brasselet-pflaum} is in progress.

%%% Local Variables:
%%% mode: latex
%%% TeX-master: "ChrPfl"
%%% End:

\section{Rational Homotopy Theory and the Main Results}

\subsection{Real homotopy groups}

Rational homotopy theory can be thought of as homotopy theory localized to the field of fractions.
Let us explain what is meant by this.  Let $X$ be a topological space.  For simplicity assume 
that $X$ is simply connected, i.e.~$\pi_1\left(X\right) = 0$.  Then each homotopy group 
$\pi_k \left(X\right)$ is a $\Z$-module. One can localize the ring $\Z$ at the set of all 
non-zero elements.  This yields the field $\Q$ and the canonical injection $\Map{i_{\Z}}{\Z}{\Q}$.  
By extension of scalars this induces a functor 
$\Map{-\otimes_{\Z}\Q}{\Z\mbox{-mod}}{\Q\mbox{-mod}}$.  The rational homotopy groups are the image 
of the groups $\pi_k\left(X\right)$ under this functor, and are denoted 
  $$\pi_k^{\Q}\left(X\right) = \pi_k \left(X\right) \otimes_{\Z}\Q.$$
In fact one can say more.  If $\K \supset \Q$ is a field extension of $\Q$, then the inclusion of 
$\Q \hookrightarrow \K$ induces a further extension of scalars functor $-\otimes_{\Q}\K$, and the 
composition of these functors is simply $-\otimes_{\Z}\K$, the extension of scalars along the 
inclusion $\Z \subset \K$.

\begin{definition}
For any field extension $\K \supset \Q$, the $\K$-homotopy groups of the simply connected space 
$X$ are defined to be 
	$$\pi_k^{\K}\left(X\right) = \pi_{k}\left(X\right)\otimes_{\Z}\K.$$
\end{definition}

\begin{remark}
This localization has the effect of killing torsion.  That is to say, if 
$\pi_k\left(X\right) \cong F \oplus T$, where $T$ is the torsion subgroup, then 
$\pi_k^{\K}\left(X\right) = F \otimes_{\Z}\K$ is the quotient of $\pi_k\left(X\right)$ by the 
torsion subgroup tensored with $\K$.

It should be noted that in some cases one can work with a space $X$ that is non-simply connected.  
In this case, rational homotopy theory requires some care, but it can be done, 
see \cite{sullivan1977infinitesimal,deligne1975real}.
\end{remark}

Unless stated otherwise, $\mathbb{K}$ will denote in this paper a field of characteristic zero.

\subsection{CDGAs and minimal Sullivan algebras}

The rational (and real) homotopy groups of a connected simply connected topological space 
homotopy equivalent to a CW complex can be computed by an algebraic model of that space. 
The models are objects in the category of commutative differential graded algebras, and two models 
describing the rational homotopy theory of the same topological space are homotopy equivalent. 
Moreover, it has been observed  by Sullivan in the ground-breaking paper 
\cite{sullivan1977infinitesimal} that among the algebraic models of the topological space there is 
one which is minimal in a certain sense, and 
that this minimal model is uniquely determined up to isomorphisms. The machinery of algebraic models 
will be used in the proof  our main results, so we briefly explain its fundamentals in this and the 
following section. We mainly follow the notation and set-up from \cite{felix2001rational}, and 
also refer to that monograph for further details. 

Note that for the following definitions, $\K$ does not need to be a field, but only a 
commutative unital ring. 
\begin{definition}
A \emph{commutative graded algebra} (CGA) over a field $\K$ is a
graded $\K$-vector space $A = \left( A^{n} \right)_{n \in \Z}$ along with a 
graded multiplication 
  $$\Map{\mu}{A^{m}\otimes_{\K} A^{n}}{A^{m+n}},\: a\otimes b \mapsto ab,$$
which is commutative, meaning that if, $a \in A^m $ and $b \in A^n$, then
  $$ ab = \left( -1 \right)^{mn} ba.$$
A CGA $A$ is called a \emph{commutative differential graded algebra} (CDGA), if $A$ is 
endowed with a \emph{differential} $d$, which means that $d$ consists of $\K$-linear maps
  $$\Map{d^n}{A^n}{A^{n+1}}, \: a\mapsto da,$$
for each $n\in \Z$, satisfies $d^2 = 0$, and the following property: For $a \in A^m$ and $b\in A$
	$$d\left(ab\right) = \left(da\right)b + \left( -1 \right)^{m }a \left(db\right).$$
That is to say that $d$ is a degree $+1$ graded $\K$-derivation of $A$ into itself.

A \emph{morphism of CGAs}, $\Map{\phi}{A}{B}$ is a collection of $\K$-linear maps 
$\Map{\phi_i}{A_i}{B_i}$ 
preserving the graded multiplication. If in addition $A$ and $B$ are CDGAs, and  
$\phi_{i+1}d_i = d_i \phi_{i}$ for every $i \in \Z$, the map 
$\Map{\phi}{A}{B}$ is called a \emph{morphism of CDGAs}.

\end{definition}

\begin{definition}
A CDGA $A = \left( A^n \right)_{n \in \Z}$ is called connected, if $A^n = 0$ for all $n < 0$ and 
$A^0 = \K$. Similarly, $A$ is called simply connected, or $1$-connected, if $A$ is connected 
and $A^1 = 0$. The CDGA (or CGA) $A = \left( A^n \right)_{n \in \Z}$ is said to be of 
\emph{finite type}, if $\dim A^n < \infty$ for all $n\in \Z$.
\end{definition}

\begin{remark}
Obviously, CDGAs over $\K$ and their morphisms as defined above form a category 
$\cat{CDGA}_\K$. Likewise, one obtains the category $\cat{CGA}_\K$ of CGAs over $\K$.  
It is clear  that there 
is a forgetful functor from $\cat{CDGA}_\K$ to $\cat{CGA}_\K$ that forgets about the differential $d$.
Moreover, the category $\cat{CGA}_\K$ can be regarded as a full subcategory of $\cat{CDGA}_\K$
by associating to every CGA $A$ the  CDGA having the same underlying graded algebraic structure and 
the differential $d=0$.
Furthermore, one can view the category $\K$-$\cat{ALG}$ of $\K$-algebras as a full subcategory of 
$\cat{CDGA}$ via the functor that associates to a $\K$-algebra $A$ the CDGA $A$ that has 
$A^0 = A$ and $A^k = 0$ for all $k \neq 0$.
\end{remark}

\begin{remark}
  By definition, every CDGA $A$ carries the structure of a cochain complex. 
  Hence it is clear what is meant by saying that two morphisms 
  $f,g : A \rightarrow B$ between CDGAs $A$ and $B$ 
  are \emph{homotopic}, namely if their exists a \emph{chain homotopy}  $h$ from $f$ to $g$, 
  i.e.~a sequence of linear maps $h^n : A^n \rightarrow B^{n-1}$ 
  such that $f^n -g^n = d^{n-1} h^n + h^{n+1} d^n $ for all $n\in \Z$.   
  If $f: A \rightarrow B$ and $g: B \rightarrow A$ are morphisms 
  between CDGAs $A$ and $B$ such that $g \circ f$  is homotopic to $\id_A$ and 
  $f \circ g$ homotopic to $\id_B$, then $A$ and $B$ are said to be \emph{homotopy equivalent}.
  Both $f$ and $g$ are called \emph{homotopy equivalences}, then. 
  If a morphism $f : A \rightarrow B $ induces an isomorphism in cohomology 
  $H^\bullet (f) : H^\bullet (A) \rightarrow H^\bullet (B)$ one says that $f$ is a 
  \emph{quasi-isomorphism}. We write
  $f: A\qism B$ or  $ A\qism B$ to denote quasi-isomorphisms of CDGAs. The CDGAs $A$ and $B$ 
  are called \emph{weakly equivalent}, if there is a sequence of quasi-isomorphisms
  \[
     A \lqism C_1 \qism \ldots \lqism C_k \qism  B
  \]
  with $C_1,\ldots , C_k$ denoting CDGAs.  We sometimes write in symbols 
  $A \lqism \ldots \qism B$ to express that $A$ and $B$ are weakly equivalent.   
  See  \cite[Sec.~12]{felix2001rational} for further details on the homotopy theory of CDGA¡¯s. 
\end{remark}

\begin{example}
Let $X$ be a topological space. Then the singular cohomology $\HS{X}{\K}{\bullet}$ with 
coefficients in $\K$, is a CGA, where the product is given by the cup product.  If the space $X$ 
is connected, resp.~simply connected, then the CGA $\HS{X}{\K}{\bullet}$ is connected, 
resp.~simply connected.
\end{example}

\begin{example}
Let $M$ be a smooth manifold. Then the de Rham complex $\Df{M}{\bullet}$ on $M$, with the wedge 
product and the exterior derivative, is a CDGA over $\R$.  Notice that even if $M$ is connected, 
or simply connected, this CDGA is not necessarily either.  However, since the cohomology of 
$\Df{M}{\bullet}$ is isomorphic to the singular cohomology of $X$ with coefficients in $\R$, 
the cohomology $\Hd{M}{\bullet}$ is a CGA which is isomorphic as a CGA to $\HS{M}{\R}{\bullet}$.  
Hence, the de Rham cohomology is connnected, resp.~simply connected, whenever $M$ is connected, 
resp.~simply connected.
\end{example}

%\begin{example}
%Given a locally closed subset $A$ of $\R^n$ the Whitney-de Rham $\Wdf{X}{\bullet}$ complex is a CDGA over $\R$.  This is because it is a quotient of the de Rham complex on an open neighborhood of $X$ by a differential graded ideal.
%\end{example}

\begin{example}
Let $V=\left(V_k \right)_{k\in\N}$ be a graded $\K$-vector space.  Let
  $$TV = \bigoplus_{n \geq 0} V^{\otimes n}$$
be the tensor algebra, with multiplication defined by concatenation of tensors.  
This is a graded algebra.  Define 
  $$SV = TV / J,$$
Where $J$ is the differential graded ideal of $TV$ generated by all elements of the form 
  $$a\otimes b - \left( -1 \right)^{mn} b \otimes a \text{ where } a \in V_m , b\in V_n, \text{ and }
    m,n.$$
Then $SV$ forms a commutative graded algebra with the multiplication induced from $TV$.
\end{example}

%\begin{remark}
%It should be noted that if $V$ is a differential graded Lie algebra over $\K$, then $S\left(V^* \right)$, where $V^*$ is the dual vector space of $V$, can be endowed with a differential that is induced by the dual of the Lie bracket in $V$, thereby making it a CDGA.  One should consult \cite{sullivan1977infinitesimal}, or \cite{deligne1975real} for more details.
%\end{remark}

\begin{definition}
A \emph{Sullivan CDGA}, or simply a \emph{Sullivan algebra}, is a CDGA of the form $SV$ for some 
graded vector space $V=\left( V_k \right)_{k\in\Z}$ over $\K$ with $V_{k}=0$ for all $k < 0$.  
Furthermore, it is required that $V$ has a basis $\left( v_\alpha \right)_{\alpha \in J}$, where $J$ 
is a partially ordered set such that for any $\beta \in J$
	$$ d v_{\beta} \in S\left(V_{< \beta}\right).$$
Here, $V_{< \beta}$ is the graded vector space spanned by $\left( v_\alpha \right)_{\alpha < \beta }$.

A Sullivan algebra is called \emph{minimal}, if the following condition is satisfied:
	$$  d\left( V \right) \subset \left(S^+ V\right)\left( S^+ V\right),$$
where $S^+ V$ consists of all the elements with nonzero tensor degree in $SV$.
\end{definition}

The minimality condition is simply requiring that the image of $d$ be in the 
image of the multiplication map when it is restricted to elements of strictly 
positive degree.

\begin{definition}
Given a CDGA, $A$, a \emph{minimal model for} $A$ is a minimal Sullivan algebra $SV$ 
along with a morphism of CDGAs, $\Map{m_A}{SV}{A}$, that induces an 
isomorphism on cohomology groups.
\end{definition}

With the above definitions it is now possible to explain why minimal Sullivan 
algebras are so important. 

\begin{theorem}[{\cite[Sec.~5]{sullivan1977infinitesimal} \& \cite[\S 12]{felix2001rational}}]
\label{minmodel}
To any connected, simply-connected CDGA, $A$, there exists a minimal model, $SV_A$, 
that is unique in the following sense:  If $B$ is a connected, simply-connected CDGA 
that is homotopy equivalent to $A$ then the minimal model $SV_B$ associated to $B$ is 
isomorphic to $SV_A$, via an isomorphism unique up to homotopy.
\end{theorem}

\subsection{Piecewise Linear Differential Forms}

In this section we describe a canonical way of associating to a topological space $X$ a CDGA 
$A_{PL}\left(X \right)$, called the CDGA of piecewise linear differential forms on $X$.  One should 
consult \cite[\S 10]{felix2001rational} for more details.

Before defining the CDGA $A_{PL}\left(X \right)$ there are several prerequisites that need to be 
discussed. Namely, one first needs to define a simplicial CDGA $A_{PL}$ which should be thought 
of as the algebra of polynomial differential forms on the standard simplices $|\Delta_{\bullet}|$ 
as subsets of $\R^{\bullet + 1}$.  Then the algebra $A_{PL}\left( X \right)$ is the collection of 
simplicial morphisms from the set of singular simplices in $X$ to $A_{PL}$.  This has the effect 
of assigning to each simplex in $X$ a polynomial differential form that is compatible with the 
face and degeneracy maps.

\begin{definition}
  The category of \emph{finite ordinal numbers} $\Delta$ is defined to have objects being ordered 
  sets $\Delta_n = \left\{0<1<2<\cdots<n\right\}$ for each $n\in\N$.  Morphisms in this category 
  are order preserving functions.
\end{definition}

\begin{definition}
  Let $\calC$ be any category.  A simplicial object in $C$ is a contravariant functor 
  $$\Map{K}{\Delta}{\calC}.$$
\end{definition}

The following proposition links the definition of simplicial objects, as defined above, with the 
more common idea of a simplicial object, namely a collection of specified objects endowed with 
face and degeneracy maps between these objects.

\begin{proposition}[{\cite[Prop.~8.1.3]{weibel1995introduction}}]
  To define a simplicial object in $\mathcal{C}$ it is sufficient to specify, for each $n\in \N$, an 
  object $K_n$ and morphisms $\Map{\partial_i}{K_{n+1}}{K_{n}}$ for $0\leq i \leq n+1$ and 
  $\Map{s_j}{K_{n}}{K_{n+1}}$ for $0\leq j \leq n$. These maps are required to satisfy the following 
  simplicial relations
  \[ 
  \begin{array}{ll}
     \partial_i \partial_j =  \partial_{j-1} \partial_i, & \text{for $i < j$}, \\
     s_i s_j =  s_{j+1} s_i , &  \text{for $i \leq j$}, \\
     \partial_i s_j   =  s_{j-1}\partial_i , & \text{for $j < i$}, \\
     \partial_i s_j   =  \id , & \text{for $i = j,j+1 $, and} \\
     \partial_i s_j   =  s_j \partial_{i-1} , & \text{for $i > j+1$} .	
  \end{array}
  \]
The maps $\partial_i$ are called face maps, and the maps $s_j$ are called degeneracy maps.
\end{proposition}

\begin{example}
Let $X$ be a topological space.  Consider for each $n \in \N$ the set of singular simplices in $X$, 
i.e.~the set
$$
  S_n \left( X \right) = \big\{ \Map{\sigma}{|\Delta_n|}{X}\mid\sigma\text{ is continuous }\big\},
$$
where $|\Delta_n|$ is the standard simplex in $\R^{n+1}$. The face and degeneracy maps are defined by 
including $|\Delta_n|$ into $|\Delta_{n+1}|$ as the $i$-th face, and by collapsing the $j$-th face of 
$|\Delta_{n+1}|$ respectively. These operations induce maps on $S_n \left( X \right)$ by precomposing 
$\sigma$ with the respective inclusion or quotient, thus defining a simplicial set.
\end{example}

The algebra $A_{PL}$ will be defined to be a simplicial CDGA.  This means that one must specify a 
CDGA for each $n \in \N$, along with the face and degeneracy maps which satisfy the above relations.

\begin{definition}\label{A_{PL,n}}
Fix a natural number $n$.  Define the graded $\R$-vector space $V = \left( V_j \right)_{j\geq0}$ with 
$V_0$ having basis $\left(t_i \right)_{i=0}^{n}$, $V_{1}$ having basis $\left( y_i \right)_{i=0}^{n}$, 
and $V_j= 0$ for $j\geq2$.  Let 
  $$A_{PL,n} = SV / I,$$	
where $I$ is the differential graded ideal generated by the elements $t_0 + t_1 + \cdots + t_n - 1$ 
and $y_0 + y_1 + \cdots + y_n$.  Define the differential $d$ by requiring that
$$ d \left( t_i \right) = y_i \mbox{, and } d \left( y_i \right) = 0. $$
\end{definition}

The above construction gives a CDGA for each $n \in \N$. To define a simplicial CDGA, one must now 
specify face and degeneracy maps. The maps $\partial_i$ and $s_j$ are defined in terms of the bases 
$(t_k)$ and $(y_k)$.  

\begin{definition}
Define the CDGA morphisms $\Map{\partial_i}{A_{PL,n+1}}{A_{PL,n}}$ and $\Map{s_j}{A_{PL,n}}{A_{PL,n+1}}$ 
for $0\leq i \leq n+1$, and $0 \leq j \leq n$, by defining them on generators as follows
$$ 
  \partial_i t_k = 
  \begin{cases}
    t_k ,& \text{ for $k<i$,} \\ 
    0, & \text{ for $k=i$,} \\ 
    t_{k-1}, & \text{ for $k>i$,}
  \end{cases}  
  \quad \text{ and }  \quad 
   s_j t_k = 
  \begin{cases}
    t_k, & \text{ for $k < j$,} \\ 
    t_k + t_{k+1}, & \text{ for $k=j$,} \\ 
    t_{k+1}, & \text{ for $k > j$.} 
  \end{cases}
$$
The simplicial CDGA $A_{PL}$ is defined to have components $A_{PL,n} = SV/I$ as in 
Def.~\ref{A_{PL,n}}, the face maps are the $\partial_i$'s, and the degeneracy maps are the $s_j$'s 
\end{definition}

The maps $\partial_i$ and $s_j$ are well defined because the algebra $A_{PL,n}$ is the quotient of the 
free algebra $SV$ by the ideal $I$, and $\partial_i$, $s_j$ preserve this ideal.  Furthermore one 
checks immediately that these maps satisfy the simplicial identities.  Therefore the above 
construction defines a simplicial CDGA $A_{PL}$, indeed.

\begin{definition}
Let $X$ be a topological space. Then the CDGA of \emph{polynomial differential forms} on $X$ is 
defined to be the set of all simplicial homomorphism between $S_{\bullet}\left(X\right)$ and 
$A_{PL,\bullet}$.  That is to say it is the set of morphisms in the category of simplicial sets
  $$A_{PL}\left(X\right) = \hom_{\Delta_{\bullet}}\left( S_{\bullet}\left(X\right) , A_{PL,\bullet} \right).$$
\end{definition}

\begin{theorem}[{\cite[\S 10 -- \S 12]{felix2001rational}}]
The assignment $X\mapsto A_{PL}\left( X \right)$ defines a contravariant functor from the 
category $\cat{TOP}_{CW}$ 
of topological spaces that are homotopy equivalent to a CW-complex to the category $\cat{CDGA}_{\R}$ 
of commutative differential graded algebras over $\R$.  Furthermore, this functor is invariant on 
homotopy classes of topological spaces in the following sense:  If $X$ is homotopy equivalent to 
$Y$ then the CDGAs $A_{PL}\left(X\right)$ and $A_{PL}\left( Y \right)$ are homotopy equivalent, and 
hence have isomorphic minimal Sullivan algebras.
\end{theorem}

In case the space under consideration is a smooth manifold, one can even say more. 

\begin{theorem}[{\cite[Thm.~11.4]{felix2001rational}}]
\label{WeakEquiMfd}
  Let $M$ be a smooth manifold. Then the de Rham complex $\Df{M}{\bullet}$ on $M$
  is weakly equivalent to  the CDGA $A_{PL}\left( M \right)$.
\end{theorem}

\begin{example}[Bloom--Herrera \cite{bloom-herrera}] 
\label{Ex:BloomHerrera}
Let $X\subset U$ be a closed semi-analytic subset of 
some open $U\subset \R^n$. Following \cite{bloom-herrera}, the sheaf $\Omega^\bullet_{\textup{dR},X}$ of 
(germs of) smooth differential forms on $X$ is defined as the quotient sheaf 
$\Omega^\bullet_{\textup{dR},X} := \big( \Omega^\bullet_{\textup{dR},U} / \calN^\bullet_{U,X}\big)_{|X}$, where 
$\Omega^\bullet_{\textup{dR},U}$ is the sheaf of smooth differential forms on $U$, and 
$\calN^\bullet_{U,X} \subset  \Omega^\bullet_{\textup{dR},U}$ the subsheaf of 
germs of smooth forms $\alpha$ such that the following holds true:
\begin{itemize}
\item for any smooth manifold $N$ and any smooth map $g: N \rightarrow U$ with $g(N)\subset X$, 
      one has $g^*(\alpha)=0$.
\end{itemize}
It follows by construction, that the exterior differential and the wedge product on 
$\Omega^\bullet_{\textup{dR},U}$ both factor to $\Omega^\bullet_{\textup{dR},X}$, hence the 
space of global sections $\Omega^\bullet_\textup{dR} (X)$ becomes a CDGA with the property 
$\Omega^0_\textup{dR}(X) = \calC^\infty (X)$. 
The question arises whether this CDGA  determines the (singular) cohomology of $X$ or 
is even weakly equivalent to the CDGA $A_{PL}\left( X \right)$. In fact, neither is the case as the following 
example by Bloom--Herrera \cite{bloom-herrera} shows. 

Consider the real analytic function $f: \R  \rightarrow \R^2$, $y \mapsto (y^5,y^6 +y^7)$, and let 
$U\subset \R$ be a connected neighborhood of $0$. After possibly shrinking $U$, the image $X := f(U)$ is an 
analytic space, irreducible and contractible. Since $X$ is one-dimensional 
$d\big(\Omega^1_\textup{dR} (X)\big)=0$. However, as Bloom--Herrera \cite{bloom-herrera} have proved,
$\Omega^0_\textup{dR} (X) \overset{d}{\longrightarrow} \Omega^1_\textup{dR} (X)$ is not surjective,
hence the Poincar\'e Lemma does not hold in this case. 
This entails in particular that Theorem \ref{WeakEquiMfd} is in general not true for the CDGA of smooth 
differential  forms on a semi-analytic set. 
\end{example}

\begin{definition}
The \emph{minimal model for $X$} is the minimal Sullivan algebra $SV_X$ of $A_{PL}\left(X\right)$ 
along with the CDGA morphism $\Map{\phi_X}{SV_X}{A_{PL}\left( X \right)}$.  
\end{definition}

\begin{remark}
If two spaces $X$ and $Y$ are homotopy equivalent then their minimal models $SV_X$ and $SV_Y$ 
coincide up to a unique isomorphism.
\end{remark}

The connection between minimal models and real homotopy theory shows up via 
the graded 
$\R$-vector space $V_X$ on which the minimal model of $X$ is defined.

\begin{theorem}[{\cite[Thm.~8.1 (iii)]{sullivan1977infinitesimal} \& \cite[Thm.~15.11]{felix2001rational}}]
\label{rhgroups}
The real homotopy groups of a connected simply connected space $X$ whose real singular homology 
$H_\bullet (X,\R)$ is of finite type are determined by $V_X$.  
In fact
	$$V_{X,\bullet} \cong \hom_{\R} \left( \pi_{\bullet} \left( X\right) , \R \right).$$ 
\end{theorem}

\begin{remark}
It should be noted that the above proposition holds when $\R$ is replaced by $\Q$, but since we are 
only interested in the real homotopy type of $X$ it is sufficient to only work over $\R$.
\end{remark}

\subsection{The Whitney--de Rham Complex}

Now that it is clear what is required to determine the real homotopy type of a space we introduce the main object of study; namely the Whitney--de Rham complex.  This complex is a CDGA over the real numbers.  In Section \ref{proof} we will prove that the Whitney--de Rham complex determines the real homotopy type of a semi-analytic set.

After the basic definitions we will take a sheaf theoretic point of view, and thus realize the Whitney--de Rham complex as the global sections of an appropriately defined complex of sheaves. 

Appendix \ref{appendix} is dedicated to an in depth study of the algebra of Whitney functions.  Here we will only define what is necessary for our purposes.

Let $X,U\subset \R^n$ be such that $U$ is open, and $X$ is closed in $U$.  Let $\Sf{U}$ denote the algebra of smooth functions on $U$, and $\Df{U}{\bullet}$ denote the de Rham complex of differential forms on $U$.  Recall that $\Df{U}{0} = \Sf{U}$.

\begin{definition}
The ideal of smooth functions on $U$ flat on $X$ is 
   $$\Jf{X}{U} = \left\{ f \in \Sf{U} \st D f |_{X} = 0,\mbox{ for all }D \in \calD \left(U\right) \right\},$$
where $\calD\left(U\right)$ is the space of all differential operators on $U$.
\end{definition}

\begin{remark}
The set $U$ is an open subset of $\R^n$, so if one chooses coordinates $\left( x_1, \ldots , x_n \right)$ on 
$U$, then $\calD\left(U\right)$ is generated as an algebra under composition of operators by the linear 
differential operators $\frac{\partial}{\partial x_1},\ldots , \frac{\partial}{\partial x_n}$ .
\end{remark}

\begin{remark}
Recall that $\Df{U}{\bullet}$ is a CDGA with multiplication defined by the wedge product
	$$\Map{ \wedge }{\Df{U}{\bullet}\otimes_{\Sf{U}}\Df{U}{\bullet}}{\Df{U}{\bullet}}.$$
One may restrict this multiplication to $\Jf{X}{U} \subset \Df{U}{0}$ in the first coordinate.  
Since $\Jf{X}{U}$ is an ideal of $\Sf{U}$, the image of this restricted multiplication map 
$\Jf{X}{U}\Df{U}{\bullet}$ is a differential graded ideal in $\Df{U}{\bullet}$.
\end{remark}

\begin{definition}
The differential graded ideal 
  $$\Rdf{X}{U}{\bullet} = \Jf{X}{U}\Df{U}{\bullet},$$
is called the complex of differential forms on $U$, flat on $X$.  
It's cohomology, $\Hr{X}{U}{\bullet}$, is called the relative Whitney--deRham cohomology.
\end{definition}

\begin{definition}
The CDGA of Whitney--de Rham differential forms on $X$ is the quotient 
  $$\Wdf{X}{\bullet} = \Df{U}{\bullet}/ \Rdf{X}{U}{\bullet}.$$
The Whitney--de Rham cohomology of $X$, $\Hw{X}{\bullet}$, is defined to be the cohomology of 
this complex.
\end{definition}

\begin{remark}
It appears that the definition of the CDGA $\Wdf{X}{\bullet}$ depends on the choice of the open set 
$U$ containing $X$.  This turns out not to be true, as the following proposition shows.
\end{remark}

\begin{proposition}\label{canonical}
The definition of $\Wdf{X}{\bullet}$ does not depend on the choice of open set $U$ containing 
$X$, up to isomorphism.
\end{proposition}

\begin{proof}
For $i=1,2$, let $X\subset U_i\subset\R^n$, with $U_i$ open.  The restriction map
  $$\Map{\rho_i}{\Df{U_i}{\bullet}}{\Df{U_1\cap U_2}{\bullet}}$$
descents to an isomorphism
  $$\Df{U_i}{\bullet}/\Rdf{X}{U_i}{\bullet} \cong 
    \Df{U_1 \cap U_2}{\bullet}/\Rdf{X}{U_1 \cap U_2}{\bullet}.$$
\end{proof}

\begin{remark}
Proposition \ref{canonical} means that $\Wdf{X}{\bullet}$ depends only on a ``formal neighborhood'' 
of $X$ in $\R^n$.  The main problem with this is that a continuous map $\Map{f}{X}{Y}$ between 
locally closed subspaces $X\subset R^n$, and $Y\subset \R^m$ does not necessarily induce a morphism 
of CDGAs between $\Wdf{Y}{\bullet}$ and $\Wdf{X}{\bullet}$, as one would hope.  A sufficient 
condition for this to be true would be that $f$ is the restriction of a smooth function, 
$\Map{F}{U}{V}$, between open sets $X \subset U$, and $Y \subset V$.  It is possible to define the 
appropriate category, $\cat{EP}$, of Euclidean pairs, whose objects are pairs, $\left(X,U\right)$, 
with $X\subset U \subset \R^n$ , so that that $U$ is open, and $X$ is closed in $U$, in such a way 
that that $\Wdf{-}{\bullet}$ is a functor from $\Map{\Wdf{-}{\bullet}}{\cat{EP}}{\cat{CDGA}_{\R}}$.  
The details of this will be left for another paper.
\end{remark}

\subsection{Semi-Analytic Sets}

An analytic set, resp.~algebraic set, is locally defined as a subset of some ambient affine space 
by the vanishing of a collection of analytic, resp.~algebraic, functions.  A semi-analytic set, 
resp.~semi-algebraic set, is defined instead by a collection of analytic, resp.~algebraic, 
inequalities. In Section \ref{sa sets} we give the precise definition of semi-analytic 
and subanalytic subsets of a real analytic space and state some of the most important properties 
of such sets. Here it is sufficient to say that a real semi-analytic subset of a real analytic 
manifold $M$ locally around each of its points has the form
\[
  \bigcup_{i=1}^p \big\{ x \in V  \mid \: g_{ i j } \left( x \right) = 0 \: \& \: f_{ i j } (x) > 0 
  \text{ for $j=1,\ldots , q $}\big\},
\]
where $V \subset M$ is an open neighborhood of the point, and the  $f_{ij}$, $g_{ij}$
with $i= 1,\ldots , p$, $j=1,\ldots , q$ denote (finitely many) real analytic functions on $V$.

\subsection{Main Results}
\label{Section:MainResults}
We are now in a position to state the main results of the paper.  There are two of these, each is 
a direct consequence of two technical propositions.  Here we will state both main theorems and the 
two technical propositions.  We will then show how both main theorems can be deduced from the 
propositions.  The proofs of the technical propositions are somewhat involved and will be
postponed to  Sections \ref{sa sets} and  \ref{proof}.

\begin{proposition}
\label{retract}
  Every semi-analytic subset $X$ of a real-analytic manifold $M$ has the homotopy type of a CW 
  complex. Moreover, there is an open set $U\subset M$, with $X\subset U$ closed, such that $X$ 
  is a deformation retract of $U$.  Furthermore, this open set $U$ can be chosen so that there is 
  a smaller open set $V$ with $X \subset V \subset \overline{V} \subset U$
  and which has the property that $X$ is a deformation retract of $V$, too.  
  If $X$ is compact, 
  then $V$ can be chosen so that $\overline{V}$ is compact as well.
\end{proposition}

\begin{proposition}\label{lemma}
  Let $X$ be a semi-analytic subset of a real-analytic manifold $M$, and $U \subset M$ open, such 
  that $X$ is a deformation retract of $U$. Then the quotient map defining the CDGA 
  $\Wdf{X}{\bullet}$  defines a quasi-isomorphism of CDGAs
	$$\Df{U}{\bullet} \simeq \Wdf{X}{\bullet}.$$
\end{proposition}

\begin{theorem}[Main Theorem 1]
Let $X$ be a semi-analytic subset of a real-analytic manifold $M$.
Then the  Whitney--de Rham cohomology of $X$ is isomorphic  as a CGA 
to the singular cohomology of $X$ with coefficients in $\R$:
		$$\Hw{X}{\bullet} \cong \HS{X}{\R}{\bullet}.$$
\end{theorem}

\begin{theorem}[Main Theorem 2]
Let $X$ be a connected simply connected semi-analytic set in a real-analytic manifold $M$. 
Then the Whitney--deRham complex determines the real homotopy type of $X$.
\end{theorem}

\begin{proof}
By Proposition \ref{retract} we can choose an open neighborhood $U\subset M$ of $X$ such 
$X$ is a deformation retract of $U$. We then have the chain of quasi-isomorphisms
$$
   \Wdf{X}{\bullet} \lqism \Df{U}{\bullet} \lqism \ldots \qism A_{PL}\left(U\right) \qism 
   A_{PL}\left(X\right).
$$
Hereby, the first quasi-isomorphism is from \ref{lemma}, the weak equivalence of
$\Df{U}{\bullet}$ and $A_{PL}\left(U\right)$ is Theorem \ref{WeakEquiMfd}, and the 
last quasi-isomorphism is a consequence of  $X$ being homotopy equivalent to $U$.  
The result now follows from \ref{minmodel} and \ref{rhgroups}.
\end{proof}

%%% Local Variables:
%%% mode: latex
%%% TeX-master: "ChrPfl"
%%% End:

\section{Facts About Semi-Analytic Sets}
\label{sa sets}
In this section, we define semi-analytic sets, and review some basic properties of such 
a set.  Most of the definitions and results in this section are from Hironaka's original 
paper on semianalytic sets \cite{hironaka1973subanalytic}, see also 
\cite{LojESA,bierstone1988semianalytic}.

\begin{definition}	
  An $\R$-ringed spaces $X = \Rs{ X }{ O }$ consists of a topological spaces $| X |$
  and a sheaf of $\R$-algebras $\mathcal{ O }_{ X }$ on $| X |$. 

  A morphism $\Map{\left(f,\phi\right)}{\Rs{X}{O}}{\Rs{Y}{O}}$ between ringed spaces is 
  a pair $\left(f,\phi \right)$, where $\Map{f}{|X|}{|Y|}$ is a continuous map, 
  and $\Map{\phi}{\calO_Y}{f_*\calO_X}$ is a morphism of sheaves on $|Y|$.
\end{definition}
		
\begin{remark}
  Let $U \subset \R^n$ be open.  Denote the sheaf of real analytic functions on $U$ by 
  $\calA_U$.  Then $\Rs{U}{\calA}$ is an $\R$-ringed space with $U = |U|$.
\end{remark}
		
\begin{definition}
  A local model of real analytic spaces is an $\R$-ringed space $\Rs{ S }{ O }$ for which there 
  exist an open subset $U \subset \R^n$ and $f_1, \ldots, f_m \in \calA_U$ such that
  \[
     | S | = \left\{ x \in U  \st f_1 \left( x \right)= \cdots = f_m \left( x \right)=0 \right\}
  \]
  and 
  \[
    \mathcal{ O }_{ S } = \left( \mathcal{ A }_{ U } / \left( f_1, \cdots , f_m \right)\mathcal{ A }_{ U } \right)|_{ |S| }. 
  \]
\end{definition}

\begin{definition}
  A real analytic space $X$ is an $\R$-ringed space $X = \Rs{ X }{ O }$ such that for every point $x \in |X|$, 
  there is an open set $V \subset |X|$ containing $x$, and an isomorphism of ringed spaces 
  $$\Map{\left(f,\phi\right)}{\Rs{V}{O}}{\Rs{S}{O}},$$
  where $\calO_{V} = \calO_{X}|_{V}$, and $\Rs{S}{O}$ is a local model of real analytic spaces.  
  Such a ringed subspace $\Rs{V}{O}$ will be called an open affine subset of $\Rs{X}{O}$ containing $x$.
\end{definition}
		
\begin{definition}[cf.~{\cite[Def.~2.1]{hironaka1973subanalytic} and \cite[Def.~2.1]{bierstone1988semianalytic}}]
  A subset $A \subset |X|$ of a real analytic space $\Rs{X}{O}$ is said to be 
  \emph{semi-analytic at a point $x \in X$}, 
  if there exists an open affine subset $\Rs{V}{O}\subset\Rs{X}{O}$ containing $x$, and a finite number of 
  $g_{ i j }, f_{ i  j } \in \calO_V$, $i=1,\ldots,p$, $j=1,\ldots ,q$, such that
  \[
    A \cap |V| = \bigcup_{i} \big\{ x \in |V|  \mid g_{ i j } ( x ) = 0 \: \& \:  f_{ij} (x) > 0, 
   \text{ for $j=1,\ldots ,q$} \big\}.
  \]
  The set $A \subset |X|$ is said to be \emph{semi-analytic} in $|X|$, if it is semi-analytic at every point $x \in A$.
\end{definition}
	
\begin{definition}[cf.~{\cite[Def.~3.1]{bierstone1988semianalytic}}]
  A subset $A \subset |X|$  of a real analytic space $\Rs{X}{O}$ is called \emph{subanalytic at a point $x \in X$}
  if there exists an an isomorphism  $\Map{\left(f,\phi\right)}{\Rs{V}{O}}{\Rs{S}{O}}$ 
  from an open affine subset $\Rs{V}{O}\subset\Rs{X}{O}$ containing $x$
  onto a local model of analytic spaces $\Rs{S}{O} \subset (U, {\calA}_U )$ with $U \subset \R^n$ open, 
  such that $f(A \cap |V|) \subset |S|$ is the projection of a relatively compact semi-analytic set, i.e.~such that
  $f(A \cap |V|) = \pi (B)$, where $B \subset \R^{n+m}$ is  a relatively compact semi-analytic subset 
  and $\pi : \R^{n+m} \to \R^n$ the canonical projection (onto the first $n$ coordinates).  
  The set $A \subset |X|$ is said to be \emph{subanalytic} in $|X|$, if it is subanalytic at every point $x \in A$.
\end{definition}

\begin{remark}
  It is straightforward to check that semi-analytic and subanalytic sets in $X$ are preserved under taking 
  finite union, finite intersection, and the set difference of any two. Moreover, projections of
  subanalytic sets $A \subset \R^{m+n}$ to $\R^n$ are subanalytic. 
  By definition, a semi-analytic set is subanalytic, but the converse does not hold, in general,
  by Osgood's example \cite[Ex.~2.14]{bierstone1988semianalytic}. 
  See \cite{hironaka1973subanalytic,bierstone1988semianalytic} for details 
  and further properties of semi-analytic and subanalytic sets. 
\end{remark}		

The main fact about semi-analytic sets that will be used is Hironaka's embedded 
desingularization theorem \cite[Prop.~2.4]{hironaka1973subanalytic}.  
Before stating the theorem there is some less familiar terminology that should be 
reviewed.
			
\begin{definition}[cf.~{\cite[Def.~5.4]{bierstone1988semianalytic}}]
\label{quadrants}
  Let $U \subset X$ be a coordinate domain of a smooth analytic space $X$,
  and $\Map{(z_1,\ldots ,z_n)}{U}{\R^n}$ be real analytic 
  coordinates.  Given a partition of $\left\{1,2,\ldots, n \right\}$
  into pairwise disjoint subsets sets $\calI_0, \calI_+$, and $\calI_-$, the 
  \emph{$(\calI_0,\calI_+,\calI_-)$-quadrant of $U$ with respect to the 
  coordinates $z_i$} is defined to be
  \[
  \begin{split}
    Q_{\calI_0,\calI_+,\calI_-} & = \\ 
     & \hspace{-3em} = \big\{ x \in U \mid z_i\left(x\right) = 0
    \text{ for $i\in \calI_0$}, 
    z_j\left(x\right) > 0 \text{ for $j\in \calI_+$},
    z_k\left(x\right) < 0 \text{ for $k\in \calI_-$}
    \big\}.
  \end{split}
  \]
  A subset $Q \subset U$ is said to be a \emph{union of quadrants} 
  if it is of the form
  $$ 
    Q = \bigcup_{(\calI_0,\calI_+,\calI_-)\in \calP} Q_{\calI_0,\calI_+,\calI_-},
  $$
  where the union is taken over some collection $\calP$ of partitions of
  $\left\{1,2,\ldots,n\right\}$  into pairwise disjoint 
  subsets $\calI_0,\calI_+,\calI_-$.
  \end{definition}		
			
  \begin{remark}
    It is clear that a union of quadrants $Q\subset U$ is a semi-analytic subset 
    of $U$, when $U$ is considered as an open affine subset of $\Rs{\R^n}{O}$.
  \end{remark}
			  
  \begin{definition}
    A real-analytic map $\Map{\pi}{\widehat{X}}{X}$ between real analytic spaces 
    is said to be \emph{almost everywhere an isomorphism}, if there is a closed 
    real analytic subspace $S$ of $X$, such that $S$ is nowhere dense in $X$, 
    $\pi^{-1}\left(S\right)$ is nowhere dense in $\widehat{X}$, and $\pi$ induces  
    an isomorphism 
    $\widehat{X} \setminus \pi^{-1}\left(S\right)\qism X \setminus S$ 
    of real analytic spaces.
  \end{definition}
			
  \begin{theorem}\label{blowup}
    Let $Y$ be a real analytic space countable at infinity.  Let $A$ be semi-analytic subset of $Y$.  
    Then, for every $x \in A$, there exists an open $X \subset Y$ containing $x$, and a smooth analytic space $\widehat{X}$ and a proper surjective real analytic map $\Map{\pi}{\widehat{X}}{X}$, such that for every point $y \in \widehat{X}$, there exists a local coordinate system $\left( z_1 , \cdots , z_n \right)$ centered at $y$ for which the following is true:
  \begin{itemize}
  \item  
    Within some neighborhood of $y$ in $\widehat{X}$, $\pi^{-1} \left( A \cap X \right)$ is a union of quadrants with respect to the coordinates $\left( z_1 , \cdots , z_n \right)$.
  \end{itemize}
  Furthermore, when $X$ is smooth,  $\pi$ can be chosen to be an isomorphism almost everywhere, such that the set $S \subset X$ where $\pi$ is not bijective is contained in $A \cap X$. In particular, $\pi$ then induces an isomorphism $\widehat{X} \setminus \pi^{-1}\left(A\cap X\right) \qism X \setminus A \cap X$ of real analytic spaces.
			
\end{theorem}

\begin{remark}	
  A crucial fact about semi-analytic and subanalytic sets is that they can be given the structure of a 
  stratified space, cf.~\cite[Prop.~4.8]{hironaka1973subanalytic}.  In the semi-analytic case, the minimal 
  stratification is particularly useful, in that it satisfies  Whitney's condition (B), see \cite[Thm.~4.9]{MatSM}.
  For the notion of stratified spaces used here and Whitney's condition (B), we refer the interested reader 
  to \cite{MatSM}  or to \cite{pflaum2001analytic}.  
\end{remark}
	
\begin{theorem}\label{stratified}
 Let $A$ be a subanalytic subset of a smooth real-analytic manifold $X$.  Then $A$ admits a Whitney {\rm (B)} 
 stratification.  To be precise, there exists a decomposition $A = \cup_{\alpha} A_{\alpha}$ inducing a 
 stratification $\mathcal{A}$, which satisfies the following property:
 \begin{itemize}
   \item 
     The family $A_{\alpha}$ forms a decomposition of $A$ by real-analytic submanifolds of $X$, each of which is 
     subanalytic in $X$.
   \item 
     With this stratification, $A$ is a Whitney {\rm (B)} stratified space.
 \end{itemize}
\end{theorem}

\begin{remark}
  Since by the Morrey--Grauert Theorem \cite{morrey,grauert}  every smooth real-analytic manifold $X$ has an 
  analytic embedding into some $\R^n$, every subanalytic subset $A \subset X$ can be analytically embedded in 
  some $\R^n$, too. In the following, we will sometimes silently make use of this fact. 
\end{remark}

\begin{remark}
  Let $A \subset \R^n$ be a (locally closed) Whitney (B) stratified subset, and $A = \cup_{\alpha} A_{\alpha}$
  a corresponding decomposition into smooth strata. Then, by Mather's control theory  \cite[Sec.~7]{mather1970notes}
  (see also \cite[Sec.~3.6]{pflaum2001analytic}), there exist \emph{control data} for the stratification defined by
  the $  A_{\alpha} $, i.e.~a family $\big( T_\alpha \big)$, where each $T_\alpha $ is a tubular neighborhood 
  of $A_\alpha$ with projection $\pi_\alpha : T_\alpha \to A_\alpha$ and tubular function 
  $\varrho_\alpha: T_\alpha \rightarrow \R$ such that for all strata $A_\alpha , A_\beta$ with $A_\beta$ incident to 
  $A_\alpha$ the following relations hold true for all $ x \in T_\alpha \cap T_\beta$ with $\pi_\alpha (x) \in T_\beta$:
  \begin{align}
    \pi_\beta \circ \pi_\alpha (x) & = \pi_\beta (x), \\
    \varrho_\beta \circ \pi_\alpha (x) & = \varrho_\beta (x).
  \end{align}
\end{remark} 

\begin{corollary}\label{retract}
  For every subanalytic subset $A\subset \R^n$ there is an open set $U\subset \R^n$, with $X\subset U$, such that 
  $X$ is a deformation retract of $U$.  Furthermore, this open set $U$ can be chosen so that there is a smaller 
  open set $V$ with  $X \subset V \subset U$, $\overline{V} \subset U$, and such that
  $X$ is a deformation retract of $V$ too.  If $X$ is compact, then $V$ can be 
  chosen such that $\overline{V}$ is compact as well.
\end{corollary}

\begin{proof}
  Without loss of generality, we can assume that $A$ is connected, otherwise we perform the following construction 
  for each connected component seperately. 
  Since $A$ is locally closed in $\R^n$, there exists an open connected subset $W \subset \R^n$ such that $X$ is 
  a relatively closed subset of $W$.  

  Let us first assume that $\dim A < n$. Choose a Whitney (B) stratification of 
  $A$ with decomposition $A = \cup_{\alpha \in J} A_\alpha$ into the strata $A_\alpha$. Put $A_\circ := W \setminus X$,
  where we assume without loss of generality that $\circ \notin J$. Put $I = J \cup \{ \circ \}$. Then
  the decomposition $ W = \cup_{\alpha \in I} A_\alpha$ defines a Whitney (B) stratification of $W$. 
  According to the preceding remark, there exist control data for this stratification. Hence, by the (proof of) 
  \cite[Thm.~3.9.4]{pflaum2001analytic} there exists an $(n-1)$-dimensional manifold $Q$, and a proper
  continuous mapping $H: Q \times [0,1] \to W$ with the following properties:
  \begin{enumerate}
    \item
     $H\big( Q \times ]0,1[ \big) \subset A_\circ$ and $H|_{Q\times ]0,1[}$ 
     is a smooth embedding. 
  \item
     $H(Q \times \{ 0 \} ) = A$.
  \item
     The image of $H$ is a neighborhood of $A$ in $W$.
  \end{enumerate}
  But then $U := H \big( Q \times [0,1[ \big) $ and  $V := H \big( Q \times [0, 1/2[ \big) $ are neighborhoods
  of $A$ with the desired properties.  

  Now assume that $\dim A = n$. Let $A^\circ \subset A$ be the open dense stratum of $A$, and 
  $A' := A \setminus A^\circ$. Construct open neighborhoods $U'$ and $V'$ of $A'$ as before. 
  Then $U:= U' \cup A^\circ$ and $V:= V' \cup A^\circ$ have the desired properties.  
\end{proof}

%%% Local Variables:
%%% mode: latex
%%% TeX-master: "ChrPfl"
%%% End:

\section{Proof of the Main Theorem}\label{proof}

In this section we give a proof of the main technical result, Prop.~\ref{lemma}.  
In order to do this we need to shift our attention away from viewing our central 
objects as algebras, and begin viewing them as sheaves.  For example we have that 
$\Wf{X} = \Gamma\left( X ; \Ws{X} \right)$, where $\Ws{X}$ is a the sheaf of 
Whitney functions on the locally closed set $X$ of a smooth manifold $M$.  
This change of 
viewpoint allows us to work locally, and from \ref{blowup}, we know that 
locally semi-analytic sets behave in a  reasonable way. 

We begin with several definitions, most of which can be found in 
\cite{pflaum2001analytic}.  Let $M$ be a smooth manifold, and $X,U \subset M$ 
with $U$ open, $X\subset U$, and $X$ closed in $U$.  Let $\Ss{U}$ denote the 
sheaf of smooth functions on $U$.

\begin{definition}
The sheaf of Whitney functions, $\Ws{X}$, on $X$ is defined to be the sheaf 
% associated to the presheaf
	$$A \mapsto \Wf{A} = \Sf{W}/\Jf{A}{W},$$
where $A\subset X$ is an open subset of $X$, and $W \subset U$ is open, such that 
$W\cap X= A$.
\end{definition}

Similarly one can define a complex of sheaves by assigning to an open set $A$ the 
Whitney de Rham complex on $A$

\begin{definition}
The Whitney--de Rham complex of sheaves, $\Wds{X}{\bullet}$, of $X$ is the complex of 
sheaves 
% associated to the presheaf
	$$A \mapsto \Wdf{A}{\bullet} = \Df{W}{\bullet}/\Rdf{A}{W}{\bullet},$$
where $A$ and $W$ are as in the previous definition.  
\end{definition}

\begin{remark}
It should be noted that the two presheaves defined above satisfy the unique gluing 
property, hence are sheaves indeed. This is due to the Whitney Extension Theorem \ref{extension}, and the fact $\Ss{U}$ satisfies this property.  Furthermore, one can see 
that $\Ws{X} = \Wds{X}{0}$, and that by Prop.~\ref{canonical} the definitions are 
independent of the choice of the open set $W \subset M$ such that $A = W \cap X$.
\end{remark}

\begin{lemma}\label{fine}
For each $i\geq0$ the sheaf $\Wds{X}{i}$ is a fine sheaf of $\Ws{X}$-modules.
\end{lemma}

\begin{proof} 
  This is because $\Ss{U}$ is a fine sheaf.  
  See \cite[Sec.~1.5.4]{pflaum2001analytic} for more details.
\end{proof}

\begin{remark}
Notice that there is a surjective morphism of sheaves of $\R$-algebras
	$$\Map{J}{\Ss{U}}{\Ws{X}},$$
which induces the short exact sequence of chain complexes 
	$$0\rightarrow \Rdf{U}{X}{\bullet} \rightarrow \Df{U}{\bullet} \stackrel{J}{\rightarrow} \Wdf{X}{\bullet} \rightarrow 0,$$
in which $J$ is a CDGA morphism.
\end{remark}

We now restrict our attention to the case when $X$ is a semi-analytic set.  The main fact about semi-analytic sets that we will use is the embedded desingularization theorem by Hironaka.  One should consult section \ref{sa sets} for the precise statement of Hironka's theorem, and for the subsequent definitions.

The strategy that we employ, to prove Prop.~\ref{lemma}, is to prove that the chain complex of sheaves $\Wds{X}{\bullet}$ is a fine resolution of the locally constant sheaf $\R_X$ on $X$.  We have already seen in Lemma \ref{fine} that each of the sheaves is 
fine.  Hence we need only to prove that the complex is exact.  To do this we begin with the simplest case, when $X$ is locally a union of quadrants as in Def.~\ref{quadrants}, then via Hironaka's theorem we reduce the general case to this simpler case.

\begin{remark}
It should be noted that there is a canonical monomorphism of sheaves $\R_X \rightarrow \Ws{X}$ given by specifying constant functions.
\end{remark}

\begin{proposition}\label{poincare quadrants}
Let $X$ be a local union of quadrants, then the chain complex of sheaves, $\R_X \rightarrow \Wds{X}{\bullet}$, on $X$, is exact. 
\end{proposition}
\begin{proof}
Assume that $X \subset M$ where $M$ is an analytic manifold. For $x \in X$, choose 
a coordinate domain $U \subset M$, and 
$\ov{z}=\left(z_1,\ldots,z_n\right):U\rightarrow \R^n$ real analytic coordinates.  
By the assumption that $X$ is a local union of quadrants this can be done so that 
$\ov{z}\left(x\right) = 0$, and $\ov{z}\left(X\cap U\right)=Q$ is a union of 
quadrants, as in Definition \ref{quadrants}.  The map $\ov{z}$ induces an isomorphism
	$$\ov{z}^*:\Wdf{Q}{\bullet}\rightarrow \Wdf{X\cap U}{\bullet}.$$
In particular we have an isomorphism at the level of stalks
\begin{equation}\label{stalks}\ov{z}_x^*: \Wds{Q,0}{\bullet}\rightarrow \Wds{X,x}{\bullet}.\end{equation}
Choose an open convex neighborhood $V\subset \R^n$ containing $0$.
Define the radial homotopy 
	$$H:\: V\times I \rightarrow V, \:\left(x,t\right)\mapsto(1-t)x,$$
Obviously, $H$ is smooth, and for any $x \in Q$, one has $H\left(x,t\right)\in Q$ for all $t\in I$.  Thus for every $t \in I$ the function $H_t=H\left(-,t\right)$ preserves the ideal $\Jf{Q}{V}$.  Hence the chain homotopy $H^*$  on $\Df{V}{\bullet}$ induced
by $H$ descends to a chain homotopy on $\Wdf{V}{\bullet}$.  Furthermore, the following identity holds
	$$H^* d + d H^* = H_1^* - H_0^* \: ,$$
where $d$ here is the differential on the complex $\Wdf{Q}{\bullet}$.  Now we note that $H_1^*$ is the zero map in all degrees except zero where it acts by evaluating a Whitney function at $0$.  Similarly the map $H_0^*$ is easily seen to be the identity on $\Wdf{Q}{\bullet}$.  Thus the chain homotopy $H^*$ is a contracting homotopy and the chain complex $\Wdf{Q}{\bullet}$ has cohomology $\Hw{Q}{k}=0$ if $k \geq 0$, and $\Hw{Q}{k} = \R$ if $k=0$.  This implies that at the level of stalks the sequence $\R_{Q,0} \rightarrow \Wds{Q,o}{\bullet}$ is exact.  By \eqref{stalks} this is sufficient, and we have shown that $\R_X \rightarrow \Wds{X}{\bullet}$ is exact

\end{proof}

\begin{theorem}
Let $X$ be a local union of quadrants, then there is an isomorphism
\[
  \HS{X}{\R}{\bullet} \cong H^\bullet \left( \sWds{ X }{ \bullet } \right)= 
  \Hw{X}{\bullet}
\]
\end{theorem}
\begin{proof}Lemma \ref{poincare quadrants} proves that $\R_X \rightarrow \Wds{X}{\bullet}$ is a fine resolution of the locally constant sheaf on $X$.  Hence the desired result.
\end{proof}

\begin{corollary}
Let $X \subset \C^n$ be an algebraic set that is a local normal crossing, then $\Hw{X}{\bullet} \cong \HS{X}{\R}{\bullet}$.
\end{corollary}

\begin{remark}
Because we have the exact sequence of chain complexes
$$0\rightarrow \Rdf{U}{X}{\bullet} \rightarrow \Df{U}{\bullet} \stackrel{J}{\rightarrow} \Wdf{X}{\bullet} \rightarrow 0,$$
the above theorem proves more than the fact that $\Hw{X}{\bullet} \cong \HS{X}{\R}{\bullet}$.  It actually shows that locally $J^\bullet$ is a quasi-isomorphism of CDGA's.  Thus one may conclude that $J^\bullet$ is a quasi-isomorphism of CDGA's.  This is restated in the following corollary.
\end{remark}

\begin{corollary}
When $X \subset \R^n$ is locally a union of quadrants, and $U \subset \R^n$ is an open subset such that $X$ is a deformation retract of $U$ in $\R^n$ then the quotient map 
	$$\Map{J^\bullet}{\Df{U}{\bullet}}{\Wdf{X}{\bullet}}$$
is a quasi-isomorphism of CDGA's.  In particular, the complex 
	$$\Rdf{X}{U}{\bullet} = \Jf{X}{U}\Df{U}{\bullet}$$ is acyclic.
\end{corollary}

We now move on to the more general case.  Let us set the notation for the remainder of this section.  Let $X \subset \R^n$ be a semi-analytic set, $x \in X$, and 
$U \subset \R^n$ open, with $x \in U$.  Let $A=X\cap U$.  After possibly shrinking
$U$, there exists an open set $V \subset \R^n$, and a proper surjective analytic map 
$\Map{\pi}{V}{U}$, such that $B=\pi^{-1}\left( A \right)$ is a local union of quadrants, 
and $\Map{\pi|_{V\setminus B}}{V\setminus B}{U\setminus A}$ is an isomorphism, 
with analytic inverse $\Map{s}{U\setminus A}{V\setminus B}$.  
By shrinking $U$ again, one can achieve that $A$ is contractible to $x$ in $A$. 
Note that because $A$ and $B$ are semi-analytic sets one can shrink 
both $U$ and $V$ by Cor.~\ref{retract} even further in such a way that $A$ and $B$ 
stay invariant and such that $A$ becomes a deformation retract of $U$ and $B$ 
a deformation retract of $V$.

These data yield the following commutative diagram of chain complexes

$$\xymatrix{ 0 \ar[r] & \Rdf{B}{V}{\bullet} \ar[r]  & \Df{V}{\bullet} \ar[r]^{J^\bullet_B}  & \Wdf{B}{\bullet} \ar[r]  & 0 \\
  0 \ar[r] & \Rdf{A}{U}{\bullet} \ar[r] \ar[u]^{\pi|^*_{V\setminus B}} & \Df{U}{\bullet} \ar[r]^{J^\bullet_A} \ar[u]^{\pi^*} & \Wdf{A}{\bullet} \ar[r]  \ar[u]^{\overline{\pi}^*}& 0 }$$

Since $B$ is a local union of quadrants we know that $J^\bullet_B$ is a quasi-isomorphism, and hence $\Rdf{B}{V}{\bullet}$ is acyclic.  We need to prove that $J^\bullet_A$ is a quasi-isomorphism.  To do this we use the fact that $\pi$ is almost everywhere an isomorphism, and show that $\pi|^*_{V\setminus B}$ is an isomorphism of chain complexes.  Then
$\Rdf{A}{U}{\bullet}$ has to be acyclic as well, and we have the desired result.

\begin{definition}\label{S}
Using the map $\Map{s}{U \setminus A }{V\setminus B }$, the inverse to $\pi^*|$, define the map 
	$$\Map{S}{\Rdf{B}{V}{\bullet}}{\Rdf{A}{U}{\bullet}}$$
by the following formula, for $\omega \in \Rdf{B}{V}{i}$
$$
   S\omega \left(z\right) = 
   \begin{cases} 
      \omega\left(s\left(z\right)\right) \circ \left( D \pi (z) \right)^{\otimes i}, & 
        \text{if $z \in U \setminus A$}, \\ 
	0, & \text{if $z \in A$}. 
   \end{cases}
$$
\end{definition}
	
It is not immediately clear that the map, $S$, defined above is actually well defined,
i.e.~that it maps to smooth forms indeed.  To this end we need to verify that $S\omega\left(z\right)$ vanishes to all orders as $z$ approaches $A$.  It is sufficient to prove 
that $S|_{\Jf{A}{U}}$ is well defined, since $S$ is just the free extension of 
$S|_{\Jf{A}{U}}$ from the $0$-th piece to the entire chain complex.

\begin{lemma}
The map $S$ is well defined.
\end{lemma}

\begin{proof} 
Because both $U$ and $V$ are smooth analytic manifolds of same dimension $n$, the 
derivatives $D\pi$ and $Ds$ each have non-zero determinant (where defined) and compose to be 
the identity.  That is to say that for any $x\in U\setminus A$,
	$$ D\pi \left( s \left(x\right) \right) Ds\left(x\right) = I_n.$$
By Cramers rule this means that one can write the partial derivative of the $i$-th coordinate 
function $s_i$ of $s$ with respect to the $j$-th variable $x_j$ of $U$ as 
	$$ \frac{\partial s_i}{\partial x_j} = \frac{ \| C_{ij} \left( D\pi \left(s\left(x\right)\right)\right) \|}{\| D\pi\left(s\left(x\right)\right) \|}.$$
Here, the vertical bars $\| - \|$ stand for the determinant, and $C_{ij}\left( a \right)$ means 
the $ij$-th cofactor matrix of the matrix $a$.

Let $K\subset \R^n$ be a compact subset contained in $U$.  Since $\pi$ is proper,
$L = \pi^{-1}\left(K\right)$ is also compact.  Thus on $L$ each partial derivative of $\pi$ 
is bounded so the numerator can be bounded from above by some constant $M$, since it is a sum 
of products of partial derivatives of $\pi$.  Furthermore, $\| D\pi \|$ is an analytic function 
on $V$.  Lojasiawicz's Inequality,\cite[Thm.~4.1]{malgrange1966ideals} entails that there exist
$C,r > 0$ such that one has, for any $y \in L$,
	$$\big| \| D\pi \left( y \right) \| \big| \geq C d\left(y,B\right)^r.$$
It follows, for $0\leq i,j \leq n$, that there are constants $c,r>0$ such that for any 
$x \in K$
	$$ \left| \frac{\partial s_i}{\partial x_j}\left(x\right) \right| \leq 
           \frac{c}{d\left( s\left(x\right)  , B \right)^r }.$$
One now needs to show that there exist constants $k$ and $s$ such that for any $x\in K$
	$$
          d\left(s\left(x\right),B\right) \geq kd\left(x,A\right)^s.
        $$
This can be done in such a way that $s=1$ and 
	$$
          k = n \sup_{z\in K, 1\leq i, j\leq n} \left| \frac{\partial \pi_i}{\partial y_j}\left(z\right) \right|.
        $$
Let us provide a proof of this. 
Let $b\in B$, $y \in V$, and $\gamma$ a rectifiable curve in $V$ connecting $y$ to $b$ in $V$. 
Then $\pi\circ \gamma $ is a curve in $U$ connecting $x = \pi\left(y\right)\in U$ to 
$a = \pi\left(b\right) \in A$.  Lemma 6.11 of Bierstone--Millman \cite{bierstone1988semianalytic} says, 
not only is $\pi\circ \gamma $ rectifiable but its geodesic length $\ell (\pi \circ \gamma )$ 
can be bounded as follows
	$$ 
          \ell ( \pi \circ \gamma ) \leq
          n \, \ell( \gamma ) \, \sup_{z \in \gamma, 1\leq i, j\leq n} 
          \left|\frac{\partial \pi_i}{\partial y_j}\left(z\right)\right|.
        $$
By Cor.~\ref{retract} we know that there exists an open set $W$, containing $A$, that is homotopy equivalent to $A$.  Furthermore we can choose a compact set $Z \subset U$ in such a way that $A \subset Z \subset W$, and that for every point in $z \in Z$ there is a curve $\gamma$ that satisfies the condition that 
$\ell( \gamma ) = d\left(z,A\right)$.  As one is only interested in the behavior of the function $s$ near $A$, it is sufficient to set $K = Z$.

The partial derivatives of $\pi$ are defined and continuous on $L$, so there is a constant $M$ that 
only depends on $\pi$ and $L$ such that, for any $\gamma$ in $L$, 
$$
  \sup_{z\in \gamma, 1\leq i, j\leq n} \left| \frac{\partial \pi_i}{\partial y_j}\left(z\right)\right| \leq M.
$$  
Combining this with the previous statement, and that the length of a curve between two points is 
always at least as large as the distance between the two points, one sees that 
$$ 
  d \left( x,a\right)\leq \ell ( \pi\circ \gamma ) \leq n \, d\left(y,b\right)M,
$$
Where $\gamma$ is chosen to be the curve in $L$ connecting $y$ to $b$ with 
$\ell(\gamma) = d\left(y,b\right)$.
Now note that, since $a = \pi\left(b\right) \in A$, 
$$
  d\left( x, A\right)\leq d\left( x, a\right).
$$
This yields 
$$ 
  d\left(x,A\right)\leq n\, M \, d\left(y,b\right).
$$
Furthermore, this is true for all points $b \in B$, so if one takes the infimum over all $b\in B$, and 
writes $y=s\left(x\right)$, then one achieves the desired result
$$
  C' d\left(x,A\right) \leq  d\left(s\left(x\right),B\right),
$$
where the constant $C'>0$ only depends on $\pi$ and $K$.
 
By induction and basic calculus, one shows that for each $\alpha \in \N^n$ the function 
$\partial^{\alpha}s_i $ is a quotient $P /Q $ where $P$ is a finite sum of finite products of partial 
derivatives of $\pi|_{V\setminus B}$, all of order less than $\alpha$, and 
$Q = \| D\pi \circ  s  \|^{m_{\alpha}}$ for some integer $m_{\alpha} \geq 1$.  
Thus, one derives, for the same reasons as above,  that for some $c_\alpha >0$
$$ 
  | \partial^{\alpha}s_i \left(x\right) | \leq \frac{c_{\alpha}}{d\left(x,A\right)^{m_{\alpha}} }
  \text{ for all $x \in K$ and  $1\leq i \leq n$}.
$$
This means that for every $g \in \calJ^\infty (B,V) $ the partial derivatives of 
the function
$$
  Sg : U \rightarrow \R, \: x \mapsto 
  \begin{cases}
    g (s(x)), & \text{if $x\in U\setminus A$}, \\
    0, & \text{if $x\in A$},
  \end{cases}
$$ 
vanish to all orders on $A$.  This is because the partials of $g$ vanish 
faster than $d\left(x,A\right)^k$ for all $k \in \N$. 
\end{proof}
	
\begin{remark}
It should be noted that this is one place where it is necessary to work with 
forms that vanish to all orders on a subset.  If $\omega$ vanished only to 
finite orders, then $S\omega$ 
would not have smooth coefficients.
\end{remark}
	
\begin{proposition}
The map $$\Map{S}{\Rdf{B}{V}{\bullet}}{\Rdf{A}{U}{\bullet}}$$ is an isomorphism of commutative differential graded algebras.  Therefore  
		$$\Hr{A}{U}{\bullet}=0.$$
\end{proposition}

\begin{proof} It is easy to see that for all $\omega \in \Rdf{B}{V}{\bullet}$ 
and $\tau \in \Rdf{A}{U}{\bullet}$ the following equalities hold:
$$ 
  \pi^* S \tau = \tau \text{ and }
  S \pi^* \omega = \omega \: .
$$
Thus the desired result.  
\end{proof}
From this proposition one derives
\begin{theorem} 
The map
 $$\Map{J^\bullet_{A}}{\Df{U}{\bullet}}{\Wdf{A}{\bullet}}$$
is a quasi-isomorphism of commutative differential graded algebras.
\end{theorem}
\begin{proof}
The kernel of $J^\bullet_{A}$ is acyclic, and $J^\bullet_{A}$ is surjective.
\end{proof}

\begin{theorem}
Let $X \subset \R^n$ be a semi-analytic set.  Then $J^\bullet$ induces a 
canonical isomorphism between the Whitney--de Rham cohomology on $X$ and the 
singular cohomology of $X$ with real coefficients, 
$$\Hw{X}{\bullet} \cong \HS{X}{\R}{\bullet},$$ 
as graded algebras.
\end{theorem}

\begin{proof} 
By Cor.~\ref{retract} there exists $W \subset \R^n$, an open set containing 
the semi-analytic set $X$ such that $X$ is a deformation retract of $W$.  
Then the above results show that the quotient map 
	$$ \Map{J^\bullet}{\Df{W}{\bullet}}{\Wdf{X}{\bullet}}$$
is locally a quasi-isomorphism.  Thus it is globally a quasi-isomorphism,
since it is induced from a morphism of sheaves.  Furthermore, $W$ is a smooth 
manifold so the deRham cohomology of $W$ is isomorphic to the singular 
cohomology of $W$.  Because $W$ is homotopy equivalent to $X$, the singular 
cohomology of $W$ is isomorphic to the singular cohomology of $X$.

It is also clear that $J^\bullet$ is independent of the resolution of 
singularities of $X$.
\end{proof}

The following statement summarizes the above work.

\begin{corollary}\label{Qism}
Let $X \subset \R^n$ be a semi-analytic set, let $W \subset \R^n$ be an open 
set containing the semi-analytic set $X$ such that $X$ is a deformation 
retract of $W$.  Then the quotient map 
	$$\Map{J^\bullet}{\Df{W}{\bullet}}{\Wdf{X}{\bullet}}$$
is a quasi-isomorphism of commutative graded differential algebras.
\end{corollary}

%%% Local Variables:
%%% mode: latex
%%% TeX-master: "ChrPfl"
%%% End:

\section{An application to free loop spaces}
In this section we apply our main result to the computation of the cohomology 
of the free loop space $\calL X$ of a semianalytic set $X \subset \R^n$. Recall  
that $ \calL X = \calC \big( S^1 , X \big)$ is the space of continuous maps from the 
circle into $X$ with the compact-open topology. According to the fundamental work by
Goodwillie \cite{goodwillie}, Burghelea-Fiedorowicz \cite{burghelea-fiedorowicz}, and 
Jones \cite{jones} (cf.~also \cite[Thm.~5.54]{algebraicmodels}), there is, for every simply connected 
CW-complex $X$, a natural isomorphism of graded algebras  
\begin{equation}
\label{eq:homloopspcapl}
  \mathrm{H}^\bullet (\calL X) \cong \HH{A_{PL}\left(X\right)}{\bullet }  , 
\end{equation}
where $\HH{-}{\bullet}$ denotes the Hochschild homology functor on DGA's 
(see \cite[Sec.~5.3]{loday} for details on the Hochschild homology theory of DGA's). As 
before, the ground ring on both sides is $\mathbb{K}$, a field of characteristic zero.
Since Hochschild homology is invariant under quasi-isomorphisms of DGA's,
cf.~\cite[Thm.~5.3.5]{loday}, one can replace in Eq.~\eqref{eq:homloopspcapl}
$A_{PL}\left(X\right)$ by any quasi-isomorphic DGA. This implies that for 
a simply connected smooth manifold $M$, one can replace $A_{PL}\left(M\right)$ 
by $\Omega^\bullet (M)$, and obtains a natural isomorphism 
\begin{displaymath}
  \mathrm{H}^\bullet (\calL M) \cong \HH{ \Df{M}{\bullet} }{\bullet }  .
\end{displaymath}
Using our main result one can extend this result to semianalytic sets. 

\begin{theorem}
 Given a simply connected semianalytic set $X \subset \R^n$, there is a natural 
 quasi-isomorphism of graded algebras
 \begin{displaymath}
  \mathrm{H}^\bullet (\calL X) \cong \HH{ \Wcx{X} }{\bullet }  , 
\end{displaymath}
\end{theorem} 

\begin{proof}
  The claim follows directly from Eq.~\eqref{eq:homloopspcapl} and the fact that
  by the proof of our Main Theorem 2 the Whitney--de Rham complex $\Wcx{X}$ is 
  quasi-isomorphic as a DGA to $A_{PL}\left(X\right)$. 
\end{proof}

%%% Local Variables:
%%% mode: latex
%%% TeX-master: "ChrPfl"
%%% End:

\appendix
\section{Basic Facts About Whitney Functions}
\label{appendix}
The aim of this section is to give a basic exposition of  Whitney functions for the convenience 
of the reader. For further details, see Malgrange \cite{malgrange1966ideals}. 
There are two approaches to Whitney functions.  The first, more classical approach, uses the 
language of jets. In the second, the algebra of Whitney functions is defined as a quotient algebra 
of smooth functions by an ideal of flat functions. The goal of this section is to give both 
approaches, and show that they are equivalent by Whitney's Extension Lemma.
Moreover, we define Whitney functions in the second approach not only over closed subsets of euclidean space
but more generally over closed subsets of a smooth manifold $M$.

\subsection{The jet approach}

\begin{remark}
Let us fix the notation that will be used throughout this section.  Fix an $n$-dimensional smooth 
manifold $M$. Let $m \in \N$. Define the set 
\[
  \NI{m}{n} = \left\{ \alpha \in \N^n \st | \alpha | \leq m \right\},
\]
where $| \alpha | = \alpha_1 + \alpha_2 + \cdots \alpha_n$ 
for $\alpha = \left(\alpha_1,\cdots, \alpha_n\right)$.  For such an $\alpha$, and a 
smooth coordinate system $x: U \rightarrow \R^n$ over an open domain $U\subset M$ define the 
partial differential operator $\partial^{\alpha}_x$ (or briefly only $\partial^{\alpha}$) by
\[
  \partial^\alpha = \partial^\alpha_x = 
  \frac{\partial^{\alpha_1}}{\partial x_1^{\alpha_1}}\cdot \ldots \cdot 
  \frac{\partial ^{\alpha_n}}{\partial x_n^{\alpha_n}}.
\]
For $k \in \N$ or $k= \infty$, let $\Ckf{U}{k}$ denote the $k$-times differentiable functions on 
$U$. When $k=0$, the superscript will be omitted and $\Ckf{U}{}$ will represent the continuous 
functions on $U$.
\end{remark}

We start with defining jets on a locally closed subset $X$ of euclidean space $\R^n$. Choose an open $U\subset \R^n$ such that $X$ is relatively closed in 
$U$.  Fix $m \in \N$, and let $0 \leq k \leq m$.

\begin{definition}
   The space of \emph{$m$-jets on $X$} is defined to be
   \[
   \begin{split}
      \Jet{ X }{ m } &  =  \Ckf{ x(X\cap U) }{}^{\NI{ m }{ n} } = \\
      & = \big\{ (F_\alpha)_{ \alpha \in \NI{ m }{ n} } \mid F_\alpha \in \Ckf{ X }{} 
      \text{ for all $\alpha \in \NI{ m }{ n}$} \big\} .
   \end{split}
   \]
\end{definition}

\begin{definition}
   For $\beta \in \NI{m}{n}$, we denote the natural projection onto the $\beta$-th factor by
   $$\Map{P_{\beta}}{\Jet{X}{m}}{\Ckf{X}{}}, \: F = (F_\alpha)_{ \alpha \in \NI{ m }{ n} } \mapsto F_\beta .$$
   The \emph{$k$-th jet projection} is the function
   $$\Map{P_{k,m}}{\Jet{X}{m}}{\Jet{X}{k}}, \:  F = (F_\alpha)_{ \alpha \in \NI{ m }{ n} } 
   \mapsto \left( F_{\alpha} \right)_{ \alpha \in \NI{k}{n} }.$$
\end{definition}

\begin{remark}
   Since $P_{k,m}P_{m,l} = P_{k,l}$, for $k\leq m \leq l$, the collection $\left(\Jet{X}{k},P_{k,l}\right)$ forms an inverse system.
\end{remark}

\begin{definition}The space of infinite jets is defined to be the limit of the inverse system
	$$\Jet{X}{\infty} = \lim_{\stackrel{\longleftarrow}{m \in \N}} \left(\Jet{X}{m},P_{k,m}\right).$$
\end{definition}
For the remainder of this section we take $m \in \N$ or $m=\infty$, and $0 \leq k \leq m$.
\begin{definition}
    For a fixed $x \in X$, one defines the following functions:\\
    \begin{itemize}
    \item the \emph{$k$-th order Taylor approximation at $x$} 
    \[
    \begin{split}
       \Map{ T^{ k }_{ x } }{ \: & \Jet{ X }{ m } }{ \Sf{ U } },
       \: F \mapsto \left( y\mapsto \sum_{ \alpha \in \NI{ k }{ n } } F_{ \alpha } \left( x \right) \frac{ \left( x - y \right)^{ \alpha } }{ \alpha ! }\right),
    \end{split}
    \]
    \item the \emph{$k$-th jet function}
    \[
       \Map{ J^{ k } }{ \Sf{ U } }{ \Jet{ X }{ k } },\: 
       f  \mapsto \left( \partial^{ \alpha } f |_{ X } \right)_{ \alpha \in \NI{ k }{ n } },
    \]
    \item the \emph{$k$-th order Taylor jet at $x$} 
    \[
       \Map{ \widetilde{ T }^{ k }_{ x } }{ \Jet{ X }{ m } }{ \Jet{ X }{ k } },\: 
       \widetilde{ T }^{ k }_{ x } = J^k T^{ k }_{ x }
    \]
    \item and the \emph{$k$-th order Taylor remainder jet at $x$} 
    \[
       \Map{ R^{ k }_{ x} }{ \Jet{ X }{ m } }{ \Jet{ X }{ k } }, \:
       R^{ k }_{ x } = P_{k,m} - \widetilde{T}^{ k }_{ x }. 
    \]
    \end{itemize}
    
\end{definition}
For $F \in \Jet{ X }{ m }$, and $\beta \in \NI{ m }{ n }$, one can consider the function $\left( R^{ m }_{ x } \left( F \right)\right)_{\beta} \left( y \right)$ as 
a function on $X \times X $.  Denote this function by $\Map{r_{F,\beta}}{X\times X}{\R}$.

\begin{definition}
  A jet $F \in \Jet{ F }{ m }$ with $m\in \N$ is called a \emph{Whitney function of class $m$ on $X$}, if 
  \[
    r_{ F , \beta }\left( x , y \right) = o\left( | x - y |^{ m - | \beta | } \right), \quad (x,y) \in K \times K,
  \]
 for every compact subset $K\subset X$ and all $\beta \in \NI{ m }{ n }$. 
 The collection of all Whitney functions on $X$ of class $m$ will be denoted by $\WF{ X }{ m }$.
 
 The inverse limit 
 \[ 
   \WF{X}{\infty} := \lim\limits_{\stackrel{\longleftarrow}{m \in \N}} \left(\WF{X}{m},P_{k,m}\right)
 \]
 is called the space of \emph{Whitney function of class $\infty$ on $X$}.
\end{definition}

\begin{definition}
 For every compact $K\subset X$ and every $k\in \N$ with $k\leq m$, $m\in \N \cup \{ \infty \}$ define a seminorm $| \bullet |_{K,k} $ on 
 $J^m \left( X \right)$ by
 \[
    | F |_{K,k} = \sup \big\{ | f_{ \beta } \left( x \right) | \mid x\in K, \beta \in \NI{ k }{ n } \big\},
 \]
 and a seminorm $\| \bullet \|_{ K,k }$ on $\WF{ X }{ m }$ by
 \[
   \| F \|_{ K,k } = | F |_{K,k}  + \sup \big\{ | r_{ F , \beta } \left( x , y \right) | \mid \left( x , y \right) \in K \times K , \:
   \beta \in \NI{ k }{ n } \big\}.
 \]
\end{definition}

\begin{proposition}[cf.~Malgrange {\cite{malgrange1966ideals}}]
  The space $ \WF{X}{m}$ together with the seminorms $\| \bullet \|_{ K,k }$ has the following properties:
  \begin{itemize}
  \item For compact $X$ and $m\in \N$, the seminorm $\| \bullet \|_{ X,m }$ defines a norm on $\WF{X}{m}$.
        Together with this norm, $\left( \WF{ X }{ m }, \| \bullet \|_{ X , m } \right)$ is a Banach space.
  \item If $X$ is compact, $\WF{X}{\infty}$ together with the family of seminorms $\big( \| \bullet \|_{ X,k } \big)_{k\in \N}$
        is a Fr\'echet space. 
  \item If $X \subset \R^n$ is locally closed, and $\big( K_k \big)_{k\in \N}$ a compact exhaustion, then
        $\WF{X}{\infty}$ together with the family of seminorms $\big( \| \bullet \|_{ K_k , k } \big)_{k\in \N}$ is a Fr\'echet space. 
  \end{itemize}
\end{proposition}

\begin{definition}
  Let $X\subset \R^n$ be locally closed, and $F , G \in \WF{ X }{ m }$.  Write $F = \left( F_{ \alpha } \right)_{ \alpha \in \NI{ m }{ n } }$ and 
  $G = \left( G_{ \alpha }	\right)_{ \alpha \in \NI{ m }{ n } }$.  Define the product of $F$ and $G$ to be 
  \[ 
    FG := \left( H_{ \alpha } \right)_{ \alpha \in \NI{ m }{ n } }, \: \text{ where $ H_{ \alpha } = \sum_{ \beta + \gamma = \alpha } F_{ \beta } G_{ \gamma }$}.
  \]
\end{definition}

\begin{proposition}
  With the above defined product, $\WF{X}{m}$ becomes a commutative Fr\'echet algebra, and  $\Wf{X}$ is the inverse limit of 
  the Fr\'{e}chet algebras  $\WF{X}{m}$, $m\in \N$.
\end{proposition}

\subsection{The quotient by flat functions approach}
Let $X \subset M$ be locally closed, and $U \subset M$ open with $X\subset U$ relatively closed. Let $m \in \N \cup \{ \infty \}$.

\begin{definition}
The ideal of smooth functions on $U$ that are \emph{$m$-flat on $X$} is
\[
\begin{split}
  &\Jfk{ X}{U}{m} =  \\ 
  & \hspace{0.5em} =\big\{ f \in \Sf{U} \mid  \partial_x^{\alpha} f |_{X\cap V} = 0 \: \text{ for all coordinates $x: V \rightarrow \R^n$ and 
                 $\alpha \in \NI{m}{n}$} \big\}.
\end{split}
\] 
\end{definition}

\begin{definition}
The \emph{algebra of Whitney functions of order $m$ on $X$} is defined to be the quotient
$$\WF{X}{m} = \Ckf{U}{m}/ \Jfk{X}{U}{m}.$$
\end{definition}

\subsection{Equivalence of the two approaches}

In this section, assume that $X \subset \R^n$ is locally closed, and that $X$ is relatively closed in the open subset $U\subset \R^n$.
For each $m \in \N \cup \{ \infty \}$ consider the map defined above as
\[
  \Map{J^m}{\calC^m (U) }{\Jet{X}{m}}, \: f \mapsto \left( \partial^{\alpha} f |_{X} \right)_{\alpha \in \NI{m}{n}}.
\]
By Taylors Theorem, each jet $\left( \partial^{\alpha} f |_{X} \right)_{\alpha \in \NI{k}{n}}$ with $k\in \N$ and $k\leq m$ satisfies the condition required to 
be a Whitney function of order $k$.  Hence, one can consider $J^m$ as a map 
\[
   \Map{J^m}{\calC^m (U) }{\WF{X}{m}}.
\]
This map is clearly an algebra homomorphism whose kernel is exactly the ideal $\Jfk{X}{U}{m}$.

\begin{theorem}[Whitney Extension Lemma, {\cite[Thm.~3.2 and Thm.~4.1]{malgrange1966ideals}}]\label{extension} 
  For each $m \in \N \cup \{ \infty \} $ the following sequence  is a short exact sequence:
  \[
    0\rightarrow\Jfk{X}{U}{m} \rightarrow \calC^m (U) \stackrel{J^m}{\rightarrow} \WF{X}{m} \rightarrow 0.
  \]
  If $m$ is finite, this sequence has a continuous linear splitting.
\end{theorem}

 The Whitney Extension Lemma implies that the ``classical'' and the quotient approach lead to equivalent definitions of $\WF{X}{m}$.

%%% Local Variables:
%%% mode: latex
%%% TeX-master: "ChrPfl"
%%% End:
\nocite{MatSM,LojESA}
\bibliography{refs}
\end{document}